\documentclass[12pt]{amsart}
\usepackage{amssymb}
\usepackage[colorlinks=true,pagebackref]{hyperref}
\usepackage[dvipsnames,usenames]{color}
\hypersetup{linkcolor=RawSienna,anchorcolor=BurntOrange,citecolor=OliveGreen,filecolor=BlueViolet,menucolor=Yellow,urlcolor=OliveGreen}

\usepackage{amscd}

\renewcommand{\phi}{\varphi}
\renewcommand{\epsilon}{\varepsilon}
\renewcommand{\theta}{\vartheta}

\def\cF{\mathcal{F}}

\def\cO{\mathcal{O}}

\def\frm{\mathfrak{m}}
\def\frn{\mathfrak{n}}

\newcommand{\llbracket}{[\negthinspace[}
\newcommand{\rrbracket}{]\negthinspace]}

\newcommand{\llparenthesis}{(\negthinspace(}
\newcommand{\rrparenthesis}{)\negthinspace)}

\DeclareMathOperator{\Hom}{Hom}

\newtheorem{lemma}{Lemma}[section]
\newtheorem{theorem}[lemma]{Theorem}

\theoremstyle{definition}

\theoremstyle{remark}
\newtheorem*{remark*}{Remark}
\newtheorem*{note*}{Note}

\begin{document}

\title{Generically finite morphisms and formal neighborhoods of arcs}

\author[L.~Ein]{Lawrence~Ein}
\address{Department of Mathematics, University of
Illinois at Chicago\\ 851 South Morgan Street (M/C 249), Chicago, IL
60607-7045, USA} \email{{\tt ein@math.uic.edu}}

\author[M. Musta\c{t}\v{a}]{Mircea~Musta\c{t}\u{a}}
\address{Department of Mathematics, University of Michigan,
Ann Arbor, MI 48109, USA}
\email{{\tt mmustata@umich.edu}}

\begin{abstract}
Let $f\colon X\to Y$ be a morphism of pure-dimensional schemes of the
same dimension,  
with $X$ smooth. We prove that
if $\gamma\in J_{\infty}(X)$ is an arc on $X$ having finite order
$e$ along the ramification subscheme $R_f$ of $X$, and if its image
$\delta=f_{\infty}(\gamma)$
 on $Y$ does not lie in $J_{\infty}(Y_{\rm sing})$,
 then the induced map $T_{\gamma}J_{\infty}(X) \to
T_{\delta}J_{\infty}(Y)$ is injective, with a cokernel of dimension
$e$. In particular, if $Y$ is smooth too, and if we denote by
$\widehat{J_{\infty}(X)_{\gamma}}$ and
$\widehat{J_{\infty}(Y)_{\delta}}$ the formal neighborhoods of 
$\gamma\in J_{\infty}(X)$ and $\delta\in J_{\infty}(Y)$, then the induced morphism
$\widehat{J_{\infty}(X)_{\gamma}}\to
\widehat{J_{\infty}(Y)_{\delta}}$ is a closed embedding of
codimension $e$.

\end{abstract}

\thanks{The first author was partially supported by NSF grant
DMS-0700774, and the second author was partially supported by NSF
grant DMS-0758454 and by a Packard Fellowship}

\maketitle

\markboth{L.~Ein and M.~Musta\c{t}\u{a}}{Generically finite
morphisms and formal neighborhoods of arcs}

\section{Introduction}

All schemes and all scheme morphisms are defined over
 an algebraically closed field $k$ of arbitrary
characteristic. $X$ and $Y$ denote schemes of finite type over $k$.
We are interested in local questions, hence we may and will assume
that $X$ and $Y$ are affine.

The scheme of arcs of $X$ is denoted by $J_{\infty}(X)$. For every
$k$--algebra $A$, we have a natural bijection
$${\rm Hom}({\rm Spec}\,A,J_{\infty}(X))\simeq
{\rm Hom}({\rm Spec}\,A\llbracket t\rrbracket ,X).$$ In particular,
a $k$--valued point of $J_{\infty}(X)$ corresponds to an arc
$\gamma\colon {\rm Spec}\,k\llbracket t\rrbracket \to X$ (we stress that all arcs that we consider are $k$--valued arcs). A morphism
$f\colon X\to Y$ induces a morphism between the schemes of arcs
$f_{\infty}\colon J_{\infty}(X)\to J_{\infty}(Y)$. If $\gamma$ is an
arc in $X$, and if $Z$ is a closed subscheme of $X$ defined by the
ideal $I_Z$, then the \emph{order of vanishing} ${\rm
ord}_Z(\gamma)$ is that $e\in {\mathbf Z}_{\geq 0}\cup\{\infty\}$
such that the inverse image $\gamma^{-1}I_Z\subseteq k\llbracket
t\rrbracket$ is equal to $(t^{e})$ (with the convention that
$e=\infty$ if this ideal is zero). For an introduction to spaces of
arcs and some applications to the study of singularities of pairs,
see for example \cite{EM}.

If $f\colon X\to Y$ is a morphism of schemes of
pure dimension $n$, then we denote by $R_f$ the ramification
subscheme of $X$, defined by ${\rm Fitt}^0(\Omega_{X/Y})$, the
$0^{\rm th}$ Fitting ideal of the sheaf of relative differentials
$\Omega_{X/Y}$.  If $X$ and $Y$ are integral schemes, and
$f$ is dominant, then 
$f$ is separable if and only if $R_f$
contains no component of $X$. Note also that if both $X$ and $Y$ are
smooth and irreducible,  and $f$ is dominant, 
then $R_f$ is the degeneracy locus of the morphism
$f^*\Omega_Y^1\to\Omega_X^1$ of locally free sheaves of rank $n$. In
particular, in this case $R_f$ is either a divisor, or $R_f=X$ (when $f$ is inseparable).

If $\gamma$ is a $k$--valued point of $J_{\infty}(X)$, then we denote
by $\widehat{J_{\infty}(X)_{\gamma}}$ the formal neighborhood of $\gamma$ in $J_{\infty}(X)$,
that is, the
formal completion of $J_{\infty}(X)$ at the point $\gamma$. We put
$\widehat{{\mathbf A}^n_0}$ for the formal neighborhood of the origin in ${\mathbf
A}^n$. Note that for every arc $\gamma$ and every $e$, the closed immersion $J_{\infty}(X)\hookrightarrow J_{\infty}(X)\times {\mathbf A}^e$ given by $\alpha\to (\alpha,0)$ induces by passing to formal neighborhoods a closed immersion $\iota_{\gamma,e}\colon 
\widehat{J_{\infty}(X)_{\gamma}}\hookrightarrow
 \widehat{J_{\infty}(X)_{\gamma}}\,\widehat{\times}\,\widehat{{\mathbf
A}^e_0}$.
The following is our main result.

\begin{theorem}\label{thm1}
Let $f\colon X\to Y$ be a morphism of smooth varieties of
pure dimension $n$. If $\gamma$ is an arc on $X$ with ${\rm
ord}_{R_f}(\gamma)=e<\infty$, and if $\delta=f_{\infty}(\gamma)$,
then we have an isomorphism
$$\widehat{J_{\infty}(Y)_{\delta}}\simeq
\widehat{J_{\infty}(X)_{\gamma}}\,\widehat{\times}\,\widehat{{\mathbf
A}^e_0},$$ such that the morphism induced by $f$ corresponds to $\iota_{\gamma,e}$.
\end{theorem}

It is interesting to compare this statement with the geometric
version of the Change of Variable formula for motivic integration,
due to Kontsevich \cite{Kon} and Denef and Loeser \cite{DL}. That
result says that if $f$ is \emph{birational} and if we consider
instead the finite-level truncation morphisms $f_m\colon J_m(X)\to
J_m(Y)$, with $m\geq 2e$, then on the locus where the order of
vanishing along $R_f$ is $e$, the map $f_m$ is piecewise trivial,
with fiber ${\mathbf A}^e$. In fact, the key argument in the proof
of Theorem~\ref{thm1} is similar to the one in \cite{DL}.

We will deduce Theorem~\ref{thm1} from a statement involving only
the tangent spaces of the corresponding spaces of arcs. Note that in
this case we do not need $Y$ to be smooth.

\begin{theorem}\label{thm2}
Let $f\colon X\to Y$ be a  morphism of pure-dimensional schemes of the
same dimension,  
with $X$ smooth.
If $\gamma$ is an arc in $X$ with ${\rm
ord}_{R_f}(\gamma)=e<\infty$, and if
$\delta:=f_{\infty}(\gamma)\not\in J_{\infty}(Y_{\rm sing})$, then
the induced tangent map
$$T_{\gamma}J_{\infty}(X)\to T_{\delta}J_{\infty}(Y)$$
is injective, and its cokernel has dimension $e$.
\end{theorem}

We mention that if $X$ is a scheme of finite type over $k$, then
a theorem of Grinberg and Kazhdan \cite{GK} describes the formal neighborhood
of a $k$--valued arc $\gamma\in J_{\infty}(X)\smallsetminus J_{\infty}(X_{\rm sing})$
 as the product of the formal neighborhood of some point on a scheme of finite type, with the formal neighborhood of the origin on an infinite dimensional affine space. This result was reproved by Drinfeld in \cite{Drinfeld} using a projection of 
$X$ to an affine space. 
It would be interesting to extend the result in Theorem~\ref{thm2} to get 
more information on  $\widehat{J_{\infty}(Y)_{\delta}}$, and compare this approach to the one in \emph{loc. cit}.

Our interest in Theorem~\ref{thm1} is motivated by possible applications to the study of singularities.
The Change of Variable formula in \cite{DL} can be applied to describe invariants of singularities 
of pairs in terms of codimensions of contact loci in spaces of arcs in the presence of a resolution of singularities (see \cite{EM}). In positive characteristic however, resolutions are not available,
the only possible tool at this point being de Jong's alterations \cite{deJong}. The Change of Variable 
formula being restricted to birational morphisms, we view Theorem~\ref{thm1} as a possible replacement 
in this setting. We hope to return to this problem in the future.

\section{The tangent map of $f_{\infty}$}

In this section we prove Theorem~\ref{thm2} above. As we have
mentioned, the argument is a variation on the argument in Lemma~3.4
in \cite{DL} (see also \cite{Lo} and \cite{EM}). We start by
recalling some basic facts about Fitting ideals that we will use
(see \cite{eisenbud} for details and proofs).

Suppose that $\cF$ is a coherent sheaf on a Noetherian scheme $W$.
If $\cF$ has a free presentation
$$\cO_W^{\oplus a}\overset{\phi}\to\cO_W^{\oplus b}\to\cF\to 0,$$
then ${\rm Fitt}^r(\cF)$ is the ideal generated by the
$(b-r)$--minors of $\phi$ (by convention, this is $\cO_W$ if $r\geq
b$). In general, Fitting ideals are compatible with pull-back: if
$h\colon W'\to W$ is a morphism of schemes, then ${\rm
Fitt}^r(h^*\cF)={\rm Fitt}^r(\cF)\cdot\cO_{W'}$.
Since we work over a perfect base field, if $Y$ is a scheme of
finite type over $k$, of pure dimension $n$, then the singular locus
of $Y$ is defined by ${\rm Fitt}^n(\Omega_Y)$. 

Suppose now that $M$ is a finitely generated module over
$k\llbracket t\rrbracket$. Then
$$M\simeq k\llbracket t\rrbracket^{\oplus
m}\oplus k[t]/(t^{q_1})\oplus\cdots\oplus k[t]/(t^{q_s})$$ for some
$m, s\geq 0$ and $q_1,\ldots,q_s\geq 1$. In this case
$m=\dim_{k\llparenthesis t\rrparenthesis}M\otimes_{k\llbracket t\rrbracket}k\llparenthesis t
\rrparenthesis$, and
 ${\rm
Fitt}^m(M)=\left(t^{\sum_iq_i}\right)$.

\bigskip

\begin{proof}[Proof of Theorem~\ref{thm2}]
Let $\gamma\in J_{\infty}(X)$ be an arc lying over $x\in X$, that we
identify with the corresponding homomorphism $\cO_{X,x} \to
k\llbracket t\rrbracket$. A tangent vector at $\gamma$ to
$J_{\infty}(X)$ corresponds to a homomorphism
$\widetilde{\gamma}\colon \cO_{X,x}\to k\llbracket
s,t\rrbracket/(s^2)$ of the form $\gamma+sD$ for some
$D\colon \cO_{X,x}\to k\llbracket t\rrbracket$. The condition that
$\widetilde{\gamma}$ is a homomorphism is equivalent to $D$ being a
derivation, where $k\llbracket t\rrbracket$ is an $\cO_{X,x}$-module
via $\gamma$. In other words,
$$T_{\gamma}J_{\infty}(X)\simeq \Hom_{k\llbracket t\rrbracket}
(\Omega_{X,x}\otimes_{\cO_{X,x}}k\llbracket t\rrbracket, k\llbracket
t\rrbracket).$$ A similar description holds for
$T_{\delta}J_{\infty}(Y)$, where $\delta=f_{\infty}(\gamma)$.

Let $y=f(x)$. If we pull-back via $\gamma$ the exact sequence
$$
f^*\Omega_Y\to\Omega_X\to\Omega_{X/Y}\to 0,
$$
 we get an exact sequence
$$
\Omega_{Y,y}\otimes_{\cO_{Y,y}}k\llbracket t\rrbracket
\overset{A}\to\Omega_{X,x} \otimes_{\cO_{X,x}}k\llbracket
t\rrbracket \to \Omega_{X/Y}\otimes _{\cO_{X,x}}k\llbracket
t\rrbracket\to 0.
$$

By hypothesis $\delta\not\in J_{\infty}(Y_{\rm sing})$, hence $\delta$ maps the generic point
of ${\rm Spec}\,k\llbracket t\rrbracket$ to the nonsingular locus of $Y$. Let $n=\dim(X)=\dim(Y)$.
Since the ground field is perfect,
we deduce $\dim_{k\llparenthesis t\rrparenthesis} (\Omega_{Y,y}\otimes k\llparenthesis t
\rrparenthesis)=n$. Therefore 
$$\Omega_{Y,y}\otimes k\llbracket t\rrbracket\simeq
k\llbracket t\rrbracket^{\oplus n} \oplus k[t]/(t^{q_1})
\oplus\cdots\oplus k[t]/(t^{q_s}),$$ for some $q_1,\ldots,q_s\geq
1$. Moreover, $X$ is smooth, hence $\Omega_{X,x}\otimes k\llbracket
t\rrbracket\simeq k\llbracket t\rrbracket^{\oplus n}$. The
assumption that $\gamma^{-1}\left({\rm
Fitt}^0(\Omega_{X/Y})\right)=(t^e)$ implies that in suitable bases
we may write $A=({\rm diag}(t^{a_1},\ldots,t^{a_n}), {\mathbf
0}_{s,n})$, for some $a_1,\ldots,a_n\geq 0$ with $\sum_ia_i=e$. We
have identified the linear map $T_{\gamma}(f_{\infty})$ with
$$\Hom(A,k\llbracket t\rrbracket)
\colon \Hom\left(\Omega_{X,x}\otimes k\llbracket
t\rrbracket,k\llbracket t\rrbracket\right)\simeq k\llbracket
t\rrbracket^{\oplus n}\to \Hom\left(\Omega_{Y,y}\otimes k\llbracket
t\rrbracket,k\llbracket t\rrbracket\right)\simeq k\llbracket
t\rrbracket^{\oplus n}$$ that is given by ${\rm
diag}(t^{a_1},\ldots,t^{a_n})$. Therefore $T_{\gamma}(f_{\infty})$
is injective, and its cokernel has dimension $e$.
\end{proof}

\section{Generically finite morphisms between smooth varieties}

Suppose that $V$ is a vector space over $k$ (typically of infinite
dimension). We denote by $\widehat{{\rm Sym}(V)}$ the completion of
the symmetric algebra ${\rm Sym}(V)$ at the maximal ideal
$\oplus_{i>0}{\rm Sym}^i(V)$. This is a local ring, and if $\frm$
denotes its maximal ideal, then $\widehat{{\rm Sym}(V)}/\frm=k$ and
$\frm/\frm^2\simeq V$.

\begin{lemma}\label{lem1}
If $X$ is a smooth variety over $k$ and if $\gamma$ is an arc on
$X$, then there is a vector space $V$ such that
$$\widehat{\cO_{J_{\infty}(X),\gamma}}\simeq\widehat{{\rm Sym}(V)}.$$
\end{lemma}

\begin{proof}
If $\dim(X)=0$, then we may take $V=0$, hence we now assume that
$n=\dim(X)>0$. Suppose that $\gamma$ lies over $x\in X$. After
replacing $X$ by an open neighborhood of $x$, we may assume that $X={\rm Spec}\,R$
is affine and that we have an \'{e}tale map $h\colon X\to {\mathbf
A}^n$ with $h(x)=0$, that corresponds to $S=k[y_1,\ldots,y_n]\to R$.
We have a Cartesian diagram
\[
\begin{CD}
J_{\infty}(X) @>{h_{\infty}}>> J_{\infty}({\mathbf A}^n) \\
@VVV @VVV \\
X @>{h}>>{\mathbf A}^n
\end{CD}
\]
(see, for example, Lemma~2.9 in \cite{EM}).
Since $J_{\infty}({\mathbf A}^n)\simeq{\rm Spec}\,S[x_i\vert i\geq 1]$, we conclude that 
$J_{\infty}(X)\simeq {\rm Spec}\,R[x_i\vert i\geq 1]$, and Lemma~\ref{next_lemma} below implies 
 that the formal neighborhoods of $\gamma$ and
$h_{\infty}(\gamma)$ are isomorphic.
 Moreover, there is an automorphism of $J_{\infty}({\mathbf A}^n)$
that takes $h_{\infty}(\gamma)$ to the constant arc $\gamma_0$ over
the origin. As we have mentioned already, $J_{\infty}({\mathbf A}^n)$ is affine with
coordinate ring isomorphic to ${\rm Sym}\left(k^{({\mathbf N})}\right)$, and
such that the ideal of $\gamma_0$ corresponds to
$\oplus_{i>0}{\rm Sym}^i(k^{({\mathbf N})})$. This gives the conclusion of
the lemma.
\end{proof}.

\begin{lemma}\label{next_lemma}
Let $\phi\colon S\to R$ be an \'{e}tale morphism of $k$-algebras of finite type, and let
$\psi\colon S_{\infty}:=S[x_i\vert i\geq 1]\to R_{\infty}:=R[x_i\vert i\geq 1]$ be the induced map between the corresponding polynomial rings in countably many variables. If $\frm$ is a maximal ideal in
$R_{\infty}$ with residue field $k$, and if $\frn=\psi^{-1}(\frm)$, then we get an induced ring
isomorphism $\widehat{(S_{\infty})_{\frn}}\simeq \widehat{(R_{\infty})_{\frm}}$.
\end{lemma}

\begin{proof}
It is enough to show that for every local $k$-algebra $(T,\frm_T)$, with residue field $k$,
and such that $\frm_T^r=(0)$ for some $r\geq 1$, the induced map
\begin{equation}\label{eq_next_lemma}
{\rm Hom}_{\rm loc}((R_{\infty})_{\frm},T)\to {\rm Hom}_{\rm loc}((S_{\infty})_{\frn},T)
\end{equation}
is a bijection (here we denote by ${\rm Hom}_{\rm loc}$ the set of local morphisms
of local $k$-algebras).

Let us denote by
$\frm_0$ and $\frn_0$ the intersections of $\frm$ and $\frn$ with $R$ and $S$, respectively. 
By assumption, for every $i$ we have some $a_i\in k$ such that $x_i-a_i\in\frm$
(hence also $x_i-a_i\in\frn$). It is straightforward to check
 that giving a local homomorphism $g\colon (S_{\infty})_{\frn}
\to T$ is equivalent to giving a local homomorphism $S_{\frn_0}\to T$, and elements 
$u_i\in T$ such that $u_i-a_i\in\frm_T$ for every $i$ (more precisely, $g$ maps each $x_i$
to $u_i$). A similar description holds for the local homomorphisms $(R_{\infty})_{\frm}\to T$.
On the other hand, since $\phi$ is \'{e}tale, the induced map
$${\rm Hom}_{\rm loc}(R_{\frm_0},T)\to {\rm Hom}_{\rm loc}(S_{\frn_0},T)$$
is bijective (see, for example, \S 28 in \cite{Mat} for the basics on 
\'{e}tale ring morphisms). Putting together the above facts, we conclude that 
(\ref{eq_next_lemma}) is a bijection.
\end{proof}

\bigskip

\begin{proof}[Proof of Theorem~\ref{thm1}]
Let $(A,\frm_A)$ and $(B,\frm_B)$ denote the local rings corresponding to the
formal neighborhoods of $\gamma$ and $\delta$, respectively, and let
$u\colon B\to A$ be the local homomorphism induced by
$f_{\infty}$. It follows from Theorem~\ref{thm2} that if we denote
by $\overline{u}$ the induced $k$-linear map
$\frm_B/\frm_B^2\to\frm_A/\frm_A^2$, then $\overline{u}^*={\rm
Hom}_k(\overline{u},k)$ is injective and $\dim_k{\rm
coker}(\overline{u}^*)=e$. Therefore $\overline{u}$ is surjective
and $\dim_k{\rm ker}(\overline{u})=e$.

Let $U$ be a complement of $V:={\rm ker}(\overline{u})$ in
$\frm_B/\frm_B^2$. It follows from Lemma~\ref{lem1} that we have an
isomorphism $B\simeq \widehat{{\rm Sym}(U)}\,\widehat{\otimes}\,\widehat
{{\rm Sym}(V)}$. On the other hand, the map $v$ given
by the composition $\widehat{{\rm Sym}(U)}\to B\to A$ induces an
isomorphism $\overline{v}\colon U\to \frm_A/\frm_A^2$. Since
$\widehat{{\rm Sym}(U)}$ is complete, it follows that $v$ is
surjective. Moreover, by Lemma~\ref{lem1} we have
$A\simeq\widehat{{\rm Sym}(\frm_A/\frm_A^2)}$, hence we can find a
local homomorphism $w\colon A\to \widehat{{\rm Sym}(U)}$ that induces
$\overline{w}=\overline{v}^{-1}\colon \frm_A/\frm_A^2\to U$, and
such that $v\circ w=1_A$. We deduce that $w$ is surjective, hence an
isomorphism, and therefore $v$ is an isomorphism as well.
\end{proof}


\begin{thebibliography}{DL}

\bibitem[deJ]{deJong}
A.~J.~ de Jong,  Smoothness, semi-stability and alterations,
 Inst. Hautes \'{E}tudes Sci. Publ. Math. No. \textbf{83} (1996), 51--93.

\bibitem [DL]{DL} J.~Denef and F.~Loeser,
Germs of arcs on singular algebraic varieties and motivic
integration, Invent. Math. \textbf{135} (1999), 201--232.

\bibitem[Dr]{Drinfeld}
V.~Drinfeld,  On the Grinberg - Kazhdan formal arc theorem,
preprint available at math.AG/0203263.

\bibitem[EM]{EM}
L.~Ein and M.~Musta\c{t}\u{a}, Jet schemes and singularities,
preprint available at arxiv: math/AG.0612862.


\bibitem[Eis]{eisenbud}
D.~Eisenbud,  Commutative algebra. With a view toward algebraic
geometry, Graduate Texts in Mathematics \textbf{150},
Springer-Verlag, New York, 1995.

\bibitem[GK]{GK}
M.~Grinberg, and D.~ Kazhdan,  Versal deformations of formal arcs,
 Geom. Funct. Anal. \textbf{10} (2000),  543--555.

\bibitem[Kon]{Kon}
M.~Kontsevich, Lecture at Orsay (December 7, 1995).


\bibitem[Lo]{Lo}
E.~Looijenga, Motivic measures, in S\'{e}minaire Bourbaki, Vol.
1999/2000, Ast\'{e}risque  \textbf{276} (2002), 267--297.

\bibitem[Mat]{Mat}
H.~Matsumura,  \emph{Commutative ring theory},
Cambridge Studies in Advanced Mathematics \textbf{8}, 
Cambridge University Press, Cambridge, 1989.

\end{thebibliography}

\end{document}